\documentclass[12pt]{article}
\usepackage{amsmath}
\usepackage{amsfonts}
\usepackage{amsthm}
\usepackage{amssymb}
\usepackage{color}
\usepackage{graphicx}
\usepackage{enumerate}

\usepackage{tikz}
\usetikzlibrary{patterns}
\usetikzlibrary{decorations.pathreplacing,angles,quotes}
\usetikzlibrary{arrows}
\usetikzlibrary{calc,automata,patterns,decorations,decorations.pathmorphing}
\usetikzlibrary{fadings}

\usepackage{caption}
\usepackage{subcaption}
\usepackage{ulem}
\usepackage{graphicx}
\usepackage[colorlinks=true]{hyperref}
\hypersetup{urlcolor=blue, citecolor=violet, linkcolor=blue}
\newtheorem{theorem}{Theorem}[section]

\newtheorem{proposition}[theorem]{Proposition}

\newtheorem{problem}{Problem}

\newtheorem{corollary}[theorem]{Corollary}

\theoremstyle{definition}
\newtheorem{definition}[theorem]{Definition}
\theoremstyle{remark}
\newtheorem{remark}[theorem]{Remark}

\newcommand\remove[1]{}

\def\mA{\mathcal{A}}
\def\mB{\mathcal{B}}

\def\mm{\mathcal{M}}
\def\mn{\mathcal{N}}

\def\mf{\mathcal{F}}

\def\ms{\mathcal{S}}
\def\mcp{\mathcal{P}}
\def\mr{\mathcal{R}}
\def\mt{\mathcal{T}}
\def\mju{\mathcal{U}}

\def\f2{\mathbb{F}_2}

\def\lip{\hskip0.02cm{\rm Lip}\hskip0.01cm}

\newcommand{\ep}{\varepsilon}
\newcommand{\lin}{{\rm lin}\hskip0.02cm}

\newcommand{\acv}{\mathfrak{a}}

\newcommand{\lp}[1]{\left( #1 \right)}

\newcommand{\lc}[1]{\left\{ #1 \right\}}
\newcommand{\av}[1]{\left| #1 \right|}
\newcommand{\nm}[1]{\left\| #1 \right\|}

\newcommand{\ve}{\varepsilon}

\newcommand{\ds}{\displaystyle}

\begin{document}

\title{Finite determination for embeddings into Banach spaces: new proof, low distortion}

\author{Florin Catrina and Mikhail~I.~Ostrovskii}

\date{\today}
\maketitle

%\begin{large}

\abstract{ The main goal of this paper is to improve the result of
Ostrovskii (2012) on the finite determination of bilipschitz and
coarse embeddability of locally finite metric spaces into Banach
spaces.

There are two directions of the improvement:

(1) Substantial decrease of distortion (from about $3000$ to
$3+\ep$) is achieved by replacing the barycentric gluing by the
logarithmic spiral gluing. This decrease in the distortion is
particularly important when the finite determination is applied to
construction of embeddings.

(2) Simplification of the proof: a collection of tricks employed
in Ostrovskii (2012) is no longer needed, in addition to the
logarithmic spiral gluing we use only the Brunel-Sucheston result
on existence of spreading models.}

\section{Introduction}

\subsection{Definitions and history}

We consider classes $\mathcal{E}$ of embeddings of metric spaces
for which the notion of {\it $\mathcal{E}$-uniform embeddings} is
well defined and satisfies the condition: all restrictions of an
embedding $\phi\in\mathcal{E}$ are $\mathcal{E}$-uniform
embeddings. (We use {\it$\mathcal{E}$-embedding} instead of
{\it embedding in $\mathcal{E}$}, for short.)

Let $\mathbf{A}$ and $\mathbf{B}$ be two classes of metric spaces.
We say that $\mathcal{E}$-embeddability of spaces of  $\mathbf{A}$
into spaces of $\mathbf{B}$ is {\it finitely determined} if the
existence of an $\mathcal{E}$-embedding of $U\in \mathbf{A}$ into
$Z\in \mathbf{B}$ is equivalent to the condition that all finite
metric subspaces of  $U$ admit $\mathcal{E}$-uniform embeddings
into $Z$.

The notion of $\mathcal{E}$-uniform embeddings is well defined for
bilipschitz and coarse embeddings. We remind their definitions:

\begin{definition}\label{D:BilipCoarse}
Let $0<C<\infty$. A map $f: (A,d_A)\to (Y,d_Y)$ between two metric
spaces is called $C$-{\it Lipschitz}\label{lip} if \[\forall
u,v\in A\quad d_Y(f(u),f(v))\le Cd_A(u,v).\] A map $f$ is called
{\it Lipschitz} if it is $C$-Lipschitz for some $C<\infty$. For a
Lipschitz map $f$ we define its {\it Lipschitz
constant}\label{lipC} by
\[\lip f:=\sup_{d_A(u,v)\ne 0}\frac{d_Y(f(u),
f(v))}{d_A(u,v)}.\label{lipCE}\]

A map $f:A\to Y$ is called a {\it $C$-bilipschitz embedding} if
there exists $r>0$ such that
\begin{equation}\label{E:MapDist}\forall u,v\in A\quad rd_A(u,v)\le
d_Y(f(u),f(v))\le rCd_A(u,v).\end{equation} A {\it bilipschitz
embedding} is an embedding which is $C$-bilipschitz for some
$C<\infty$. The smallest constant $C$ for which there exists $r>0$
such that \eqref{E:MapDist} is satisfied is called the {\it
distortion} of $f$. \smallskip

A map $f:(X,d_X)\to (Y,d_Y)$ between two metric spaces is called a
{\it coarse embedding} if there exist non-decreasing functions
$\rho_1,\rho_2:[0,\infty)\to[0,\infty)$ (observe that this
condition implies that $\rho_1$ and $\rho_2$ have finite values) such that
$\lim_{t\to\infty}\rho_1(t)=\infty$ and
\begin{equation}\label{E:coarse}\forall u,v\in X~ \rho_1(d_X(u,v))\le
d_Y(f(u),f(v))\le\rho_2(d_X(u,v)).\end{equation}
\end{definition}

The corresponding notions of $\mathcal{E}$-uniform embeddings are
families of bilipschitz embeddings with uniformly bounded
distortions and families of coarse embeddings satisfying the
inequality \eqref{E:coarse} with the same functions
$\rho_1,\rho_2:[0,\infty)\to[0,\infty)$ for all members of the
family.

Recall that a metric space is called {\it locally finite} if each
ball of finite radius in it contains finitely many elements. Our
sources for Banach Spaces and Metric Embeddings are \cite{BL00}
and \cite{Ost13}. All normed vector spaces considered in this
paper are over the reals.

The main goal of this paper is to improve the result of
\cite{Ost12} (see Theorem \ref{T:FinDet} below) on the finite
determination of bilipschitz and coarse embeddability in cases
where $\mathbf{B}$ is the class of all Banach spaces and
$\mathbf{A}$ is the class of all locally finite metric spaces.

There are two directions of the improvement:

 (1) Substantial decrease of distortion (from
about $3000$ to $3+\ep$) is achieved by replacing the barycentric
gluing introduced in \cite{Bau07} by the logarithmic spiral gluing
introduced in \cite{OO19}. The decrease in the distortion is
particularly important because the finite determination is applied
(see \cite{NY18, NY22}) to construction of embeddings.

(2) Simplification of the proof: a collection of tricks employed
in \cite{Ost12} is no longer needed, in addition to the
logarithmic spiral gluing we use only the Ramsey Theorem in the
same way as Brunel and Sucheston in \cite{BS74}.

\begin{remark} The restriction that ${\bf A}$ is the class of
locally finite metric space is essential. It is well known, for
example, that all finite subsets of $\ell_2$ admit bilipschitz
embeddings into $\ell_p$ $(1\le p\le\infty)$ with uniformly
bounded distortions (it is an immediate consequence of the
Dvoretzky theorem), but $\ell_2$ does not admit a bilipschitz
embedding into $\ell_p$ $(1\le p<\infty, p\neq 2)$; see \cite[Corollary
7.10]{BL00}.
\end{remark}

Now we mention some of the steps in the development of the theory
of finite determination (although we try to mention all
contributions, we do not claim that our list is exhaustive).

\begin{enumerate}[{\bf (1)}]

\item Coarse embeddability of the class of locally finite metric
spaces into the class of Banach spaces with trivial cotype is
finitely determined \cite{Ost06}. (This means that any locally
finite metric space admits a coarse embedding into any Banach
space with trivial cotype.)

\item Bilipschitz embeddability of the infinite binary tree into
the class of all Banach spaces is finitely determined
\cite{Bau07} (based on \cite{Bou86}).

\item Bilipschitz embeddability of the class of locally finite
metric spaces into the class of Banach spaces with trivial cotype
is finitely determined \cite{BL08}. (This means that any locally
finite metric space admits a bilipschitz embedding into any Banach
space with trivial cotype.)

\item Bilipschitz embeddability of the class of locally finite
subspaces of a Hilbert space into the class of all Banach spaces
is finitely determined \cite{Ost09}. (By the Dvoretzky theorem it
means that any locally finite metric subspace of a Hilbert space
admits a bilipschitz embedding into an arbitrary infinite-dimensional Banach space.)

\item Bilipschitz embeddability of the class of locally finite
subspaces of $\mathcal{L}_p$-spaces into the class of all Banach
spaces in which the $\ell_p$ space is finitely representable is finitely determined
\cite{Bau12}.

\end{enumerate}

The final result in this direction is the following:

\begin{theorem}[\cite{Ost12}]\label{T:FinDet} Both coarse and bilipschitz
embeddability of locally finite metric spaces into Banach spaces
are  finitely determined.
\end{theorem}

In all of the papers above no efforts were made to minimize the
difference between the assumptions for the finite subsets and the
outcome for the locally finite metric spaces.
For example, the result of \cite{Ost12} proves that embeddability of
finite subsets of a metric space $M$ into a Banach space $X$ with distortions
at most 5 implies the embeddability of $M$ into $X$ with distortion at most 15,000.
On the other hand, it is natural to expect that the gap in distortion does not have to be so big.
The first steps in the direction of lowering this gap were
made in the papers: \cite{OO19}, \cite{Ost22}, \cite{COO22+}.

Some
of the obtained results (we omit more technical ones):

\begin{enumerate}[{\bf (i)}]

\item For every $\ep>0$ any locally finite metric space admits a
$(4+\ep)$-bilipschitz embedding into any Banach space with trivial
cotype \cite{OO19}.

\item For every $\ep>0$ any locally finite metric subspace of
$L_p(0,1)$ admits a $(1+\ep)$-bilipschitz embedding into $\ell_p$
\cite{OO19}.

\item For every $\ep>0$ the infinite binary tree admits a
$(4+\ep)$-bilipschitz embedding into any nonsuperreflexive space
\cite{Ost22}.

\item For every $\ep>0$ any locally finite metric subspace of a
Hilbert space admits a bilipschitz embedding into an arbitrary
Banach space with distortion $\le(1+\ep)$ \cite{COO22+}.

\end{enumerate}

\begin{remark}\label{R:Isom}
A related problem is: {\it For what Banach spaces $X$ does the
existence of isometric embeddings of finite subsets of a locally
finite metric space $A$ into $X$ imply that $A$ admits an
isometric embedding into $X$?} Numerous examples of Banach spaces
not satisfying this condition were discovered in \cite{KL08, OO19,
OO19a, OO19b}. It is easy to show by the standard ultraproduct
techniques  (use \cite[Proposition 2.21]{Ost13} for isometric
embeddings) that a Banach space $X$ satisfies the condition of the
problem if each separable subspace of each ultrapower of $X$
admits an isometric embedding into $X$. However, complete answer
to this problem seems to be unknown.
\end{remark}

It is worthwhile to mention here some works devoted to more
general than locally finite metric spaces with weaker than
bilipschitz embeddings: \cite{Bau12}, \cite{BL15}, \cite{Bau22},
\cite{Net22+}.

There are other problems, the statements of which sound somewhat
similar to finite determination problems, but the techniques used
for working on  them is quite different. One of them is the
following: {\it Let a metric space $(X,d)$ be a union of its
metric subspaces $A$ and $B$, disjoint or not. Assume that $A$ and
$B$ have a certain metric property $\mathcal{P}$. Does this imply
that $X$ also has property $\mathcal{P}$, possibly in some
weakened form?} See the papers \cite{DG07}, \cite{MN13},
\cite{MM16}, and \cite{OR22} for information on this problems and
some related problems.

\subsection{The main result and its proof plan}
\label{ss:plan}

The main step in achieving the stated above goals of the paper is
the following result. Using standard techniques
we derive from Theorem \ref{T:main} several finite determination
results (see Section \ref{S:FD}). We first remind some basic
definitions.

\begin{definition}\label{D:BM&FinRep}
The {\it Banach-Mazur distance} between Banach spaces $F$ and $G$
is defined by
$$d_{\rm BM}(F,G)=\inf\{ \|T\|\|T^{-1}\|:\  T:F\to G \text{ is an
isomorphism}\}.$$ If $F$ and $G$ are not isomorphic, we put
$d_{\rm BM}(F,G)=\infty$.

Let $Y$ and $X$ be two Banach spaces. The space $Y$ is said to be
{\it finitely representable} in $X$ if for any $\alpha>0$ and any
finite-dimensional subspace $F\subset Y$ there exists a
finite-dimensional subspace $G\subset X$ such that $d_{\rm
BM}(F,G)<1+\alpha$.
\end{definition}

\begin{theorem}
\label{T:main} Let $Y$ be an infinite-dimensional Banach space,
finitely representable in a Banach space $X$. Let $\mm$ be a
subset of $Y$ which is a locally finite metric space (with respect
to the distance induced from $Y$). Then, for any $\epsilon > 0$,
$\mm$ admits a bilipschitz embedding into $X$ with distortion at
most $(3+\epsilon)$.
\end{theorem}

First we describe the form in which the desirable embedding will
be sought. After that, we shall provide the details showing that
one can make choices in this construction leading the the desired
estimates.

For a normed vector space $X$ and $r>0$ we denote by $S_r(X)$ the
centered at $0$ sphere of radius $r$ in $X$, and by $B_r(X)$ the
closed ball centered at $0$ of radius $r$. For $0<r< R$ we use the
notation $A_{r,R}(X) = B_R(X)\setminus B_r(X) $ for the
corresponding annulus in $X$. We also use the notation $\mB_r =
B_r(Y)\cap \mm$ and $\mA_{r,R} = A_{r,R}(Y)\cap\mm$.

Since $\mm$ is locally finite, for every  $0<R<\infty$, the
subspace $Y_R$ of $Y$ spanned by $\mB_R$ is finite dimensional.
Also, because  $Y$ is finitely representable in $X$, for every
$\alpha>0$, there exists a linear operator
 $\psi_{R,\alpha}:Y_R \to X$ such that

\begin{equation}\label{E:psi} \nm{y} \le \nm{\psi_{R,\alpha} y}\le (1+\alpha)\nm{y} \mbox{ for all } y\in
 Y_R.\end{equation}

 We construct a desirable bilipschitz embedding $\Psi:\mm \to X$ as a weighted sum of $\lc{\psi_{R,\alpha}}$
 for suitably chosen $R$, $\alpha$ and $\psi_{R,\alpha}$. Note
 that $\psi_{R,\alpha}$ satisfying \eqref{E:psi}  does not have to be uniquely determined by
 $R$ and $\alpha$, and we shall use this freedom.

\section{Proof of Theorem~\ref{T:main}}

\subsection{Choices of necessary parameters and maps}
\label{ss:Cf}

We pick $\gamma,\delta,\ep\in (0,1)$. Next we pick a sequence of
radii, which will be used to choose $\psi_{R,\alpha}$, and a
sequence of weight functions $\tau_i(t)$ and $\mu_i(t)$.
\medskip

For $i\ge 1$, we construct  inductively $r_i$, $R_i$ with
\begin{equation}
\label{E:rR_ratio} \frac{1}{\delta} R_{i-1}< r_i< R_i,
\end{equation}
 and so that there
exists a continuous function $\tau_i:[0,\infty) \to [0,\pi/2]$
which is differentiable everywhere except, possibly, $r_i$ and
$R_i$, and has the properties:
\begin{enumerate}
\item \label{i:t1} $\tau_i (t) = 0$ for all $t\in [0,r_i]$, \item
\label{i:t2}  $\tau_i (t) = \pi/2$ for all $t\in [R_i, \infty)$,
\item \label{i:t3} $0\le \tau_i'(t) \le \ve/t$ for all
$t\in[r_{i},R_{i}]$.
\end{enumerate}
Note that for $r_1$ one can pick any positive value.  Next, $R_1$
has to be chosen so that $\ve \ln \frac{R_1}{r_1} \ge
\frac{\pi}{2}$ in order for a function $\tau_1$ with properties
\ref{i:t1}, \ref{i:t2}, and \ref{i:t3}, to exist. Then, $r_2$ has to
satisfy the leftmost inequality in \eqref{E:rR_ratio}, and again
$R_2$ has to be chosen so that $\ve \ln \frac{R_2}{r_2} \ge
\frac{\pi}{2}$ in order for the function $\tau_2$ with the desired
properties to exist. Inductively, we obtain the infinite sequences
$\lc{r_i}_i$, $\lc{R_i}_i$, and $\lc{\tau_i}_i$.

Using the sequence $\lc{\tau_i}_i$ we  define for all $i\ge 1$,
\[ \mu_i(t)=\cos(\tau_{i}(t))
\mbox{ and } \mu_{i+1}(t)=\sin(\tau_{i}(t)) \mbox{ for }
t\in[r_i,R_i].\]

Defined in such a way weight functions $\{\mu_i\}_{i=1}^\infty$
satisfy the conditions \eqref{i:a}-\eqref{i:e} below (see also
Figure~\ref{F:Cfunct}):

\begin{enumerate}[{\bf (a)}]

\item \label{i:a} $\mu_1$ is supported on $[0,R_1]$,

\item $\mu_i$ $(i\ge 2)$ is supported on $[r_{i-1}, R_i]$,

\item $\mu_1=1$ on $[0,r_1]$ and is decreasing on $[r_1,R_1]$,

\item For $i\ge 2$, the function $\mu_i$ is increasing on
$[r_{i-1},R_{i-1}]$, satisfies $\mu_i=1$ on $[R_{i-1},r_i]$, and
is decreasing on $[r_i,R_i]$,

\item \label{i:e}   For all $i\ge 1$,
  \[ \lp{ \lp{\mu'_i(t)}^2 + \lp{\mu'_{i+1}(t)}^2}^\frac12 \le \ve/t \ \mbox{ for all } \ t\in [r_i,R_i].\]

 \end{enumerate}

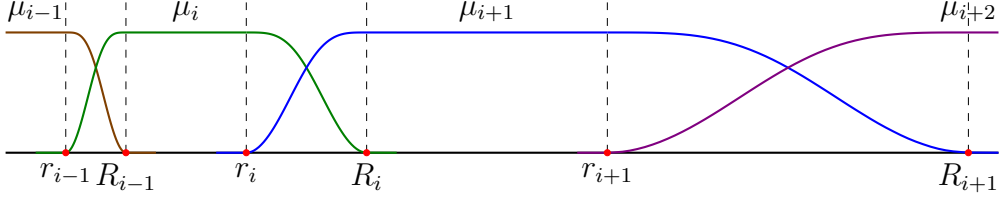
\begin{figure}
\begin{tikzpicture}[scale=0.4]

\node[above] at (1,4) {  $\mu_{i-1}$}; \node[above] at (6,4) {
$\mu_{i}$}; \node[above] at (16,4) {  $\mu_{i+1}$}; \node[above]
at (32,4) {  $\mu_{i+2}$};

\draw[thick] (0,0) -- (33,0);

\draw[very thin, dashed] (2,0)--(2,5);

\draw[very thin, dashed] (4,0)--(4,5);

\draw[very thin, dashed] (8,0)--(8,5);

\draw[very thin, dashed] (12,0)--(12,5);

\draw[very thin, dashed] (20,0)--(20,5);

\draw[very thin, dashed] (32,0)--(32,5);

 \draw[thick, color = orange!50!black, domain=2:4, smooth, variable=\x]
 (0,4) -- (2,4) --plot (\x, {4*cos((3.141/2*1/4*(5-\x)*(-2 + \x)^2) r)}) -- (5,0);

\draw[thick, color = green!50!black, domain=2:4, smooth,
variable=\x]
 (1,0) -- (2,0) --plot (\x, {4*sin((3.141/2*1/4*(5-\x)*(-2 + \x)^2) r)}) -- (8,4);

 \draw[thick, color = green!50!black, domain=8:12, smooth, variable=\x]
 plot (\x, {4*cos((3.141/2*1/32*(\x-14)*(-8 + \x)^2) r)})--(13,0) ;

 \draw[thick, color = blue, domain=8:12, smooth, variable=\x]
 (7,0) -- plot (\x, {4*sin((3.141/2*1/32*(-\x+14)*(-8 + \x)^2) r)})--(20,4) ;

 \draw[thick, color = blue, domain=20:32, smooth, variable=\x]
 plot (\x, {4*cos((3.141/2*1/54*1/16*(38 - \x)*(-20 + \x)^2) r)})--(33,0) ;

 \draw[thick, color = violet, domain=20:32, smooth, variable=\x]
 (19,0) -- plot (\x, {4*sin((3.141/2*1/54*1/16*(38 - \x)*(-20 + \x)^2) r)})--(33,4) ;

\filldraw[red] (2,0) circle (.1); \node[below] at (2,0) {
$r_{i-1}$};

\filldraw[red] (4,0) circle (.1); \node[below] at (4,0) {
$R_{i-1}$};

\filldraw[red] (8,0) circle (.1); \node[below] at (8,0) {
$r_{i}$};

\filldraw[red] (12,0) circle (.1); \node[below] at (12,0) {
$R_{i}$};

\filldraw[red] (20,0) circle (.1); \node[below] at (20,0) {
$r_{i+1}$};

\filldraw[red] (32,0) circle (.1); \node[below] at (32,0) {
$R_{i+1}$};

   \end{tikzpicture}
 \caption{Graph of coefficient functions $\lc{\mu_i}_i$ }
\label{F:Cfunct}
\end{figure}

\begin{remark} On the interval $[r_i,R_i]$, the curve
$ t\to \lp{t\cos(\tau_{i}(t)), t\sin(\tau_{i}(t))} $ resembles a
piece of a logarithmic spiral in the plane. The idea for using
weight functions with these properties originated in
\cite{OO19}. It takes advantage of the fact that  by controlling
the speed (via the parameter $\ve$) one can control the distortion
of the logarithmic spiral gluing.
\end{remark}

For each sequence $0<r_1<R_1<r_2 <R_2< \dots$ and  each locally finite
metric space $\mm$ in $Y$ we define the sequence
$\lc{Y_n}_{n=1}^\infty$ of nested ($Y_n\subset Y_{n+1}$ for all
$n\ge 1$) finite-dimensional subspaces of $Y$ as linear spans,
\[ Y_n = \lin \lp{\mB_{R_n}}.\]
Since $Y$ is finitely representable in $X$, for every
 sequence $\lc{\ve_n}_n$ of positive numbers,  there exist linear
embeddings $\psi_n :Y_n \to X$ such that for any $n\ge
1$ and any $y \in Y_n$ it holds that
\begin{equation}
\label{E:aiso}  \nm{y} \le \nm{\psi_n(y)} \le (1+\ve_n) \nm{y}.
\end{equation}
%Assume that $\prod_{i=1}^\infty (1+\ve_i) \le 1+\epsilon_1$.
We may assume that for the chosen $\gamma >0$, %{\color{red} ($\gamma =\epsilon_3$)}
the condition
$\prod_{n=1}^\infty (1+\ve_n) \le 1+\gamma$ holds.

We shall seek the desired embedding  $\Psi:\mm \to X$ in the form
\begin{equation}
\label{E:DefPsi}
\Psi(x)=\sum_{i=1}^\infty \mu_i(\|x\|)
\Psi_{i}(x),
\end{equation}
where $\lc{\Psi_i}_i$ is a subsequence of $\lc{\psi_{n}}_n$.

To show that a suitable selection of a subsequence $\lc{\Psi_i}_i$ is possible
we estimate the Lipschitz behavior of the function $\Psi$.

For each $i \ge 1$,
let $\mm_i=\lc{ x \in \mm \ :   \  R_{i-1} < \nm{x} \le r_{i+1}} \cup \lc{0}$ (with the convention that $R_0=0$).

%\subsection{Distortion estimate on $\mm_i$}%$\mA_{R_{i-1},r_{i+1}}\cup \lc{0}$}

Note that when restricted to the set $\mm_i$, the defining series for $\Psi$ reduces to
\begin{equation}
\label{E:PsiMi}
\Psi(x) = \mu_i\lp{\nm{x}} \Psi_i(x) + \mu_{i+1}\lp{\nm{x}} \Psi_{i+1}(x) .
\end{equation}
To simplify notation, let us denote:
\begin{itemize}
\item $Y_{r_{i+1}}=\lin \lp{\mB_ {r_{i+1}}}$ by $V$,
%\item $r_{i}$, $R_{i}$ by $r$, $R$,  respectively.
\item $\tau_{i}$ by $\tau$,
\item $\mu_i\lp{\nm{x}} = \cos\lp{\tau_{i}\lp{\nm{x}}}$ by $c(x) =\cos\lp{\tau\lp{\nm{x}}}$,
\item $\mu_{i+1}\lp{\nm{x}} = \sin\lp{\tau_{i}\lp{\nm{x}}}$ by $s(x)=\sin\lp{\tau\lp{\nm{x}}}$,
%$\mu_{i+1}\lp{\nm{x}} $ by $c(x)$, $s(x)$
\item $\Psi_{i}$ and  $\Psi_{i+1}$ by $E$ and $F$,  respectively.
\end{itemize}
With this notation, we will obtain  $\left. \Psi\right|_{\mm_i}= \left. T\right|_{\mm_i}$ where
$T:V\to X$ is a map which we are now going to define.

%\section{Bending and distortion estimate}
%\label{S:bend}

%Let $E$ and $F$ be two embeddings of a finite dimensional  normed linear space $V$ into a Banach space $X$.
For any angle $\theta$ define a linear operator $G_\theta$ mapping $V$ into $X$ by
\[ G_\theta x= \cos(\theta) Ex+\sin(\theta)Fx. \]
%\ \mbox{ and } \ H_\theta=- \sin(\theta) E+\cos(\theta)F.\]
We also consider the function $g:[0,\pi]\times S_1(V) \to [0,\infty)$ given by
\[ g(\theta, u) = \nm{G_\theta u} =\nm{\cos(\theta) Eu+\sin(\theta)Fu}.\]
%and
%\[ h(\theta, u) = \nm{H_\theta u} =\nm{- \sin(\theta) Eu+\cos(\theta)Fu}.\]

\begin{remark}
\label{R:beta_est}
It can be easily seen from the triangle and the Cauchy-Schwarz inequalities that
\[ g(\theta, u) \le
\av{\cos(\theta)} \nm{Eu}+\av{\sin(\theta)}\nm{Fu}
\le \lp{\nm{E u}^2 +\nm{Fu}^2}^\frac{1}{2},\]
for any angle $0 \le \theta \le \pi$ and any unit vector $u $ in $V$.
Since $1\le \nm{E u}, \nm{F u} \le 1+\gamma$, we obtain
\[g(\theta, u)\le \sqrt{2}(1+\gamma).\]
\end{remark}

\medskip
%
%Let $\ve>0$ be given, and
%assume that  $0<r < R$ are real numbers so that there exists a continuous function

\noindent
By construction, $\tau:[0,\infty) \to [0,\pi/2]$ has the properties
\begin{enumerate}[(i)]
\item $\tau (t) = 0$ for all $t\in [0,r_i]$,
\item $\tau (t) = \pi/2$ for all $t\in [R_i, \infty)$,
\item \label{i:Dtau}  $\tau$ is differentiable  for all $r_i<t <R_i$  and $0\le \tau'(t) \le \ve/t$.
%{\color{red}($\ve=\epsilon_1$)}
\end{enumerate}
Such a function $\tau$ induces a map  $T: V\to X$ via the formula
\begin{equation}
\label{E:Tdef}
 T(x) = G_{\tau(\nm{x})} x = \cos\lp{\tau(\nm{x})}Ex +\sin\lp{\tau(\nm{x})}Fx.
 \end{equation}
Note that if $E(V)\cap F(V)=\lc{0}$, then $G_\theta$ is an injective linear embedding of $V$ into $X$ for any $\theta$.
In this case $T$ will be a bilipschitz metric embedding of $V$ in $X$, but with possibly very large distortion.
However if $E(V)\cap F(V)\neq\lc{0}$ it is possible that for some angles $\theta$ the
kernel of $G_\theta$ has positive dimension.
In such a case $T$ doesn't even have an inverse defined on $T(V)$.
For example if it happens that $F=-E$, an entire sphere (of radius $t$ with $\tau(t)=\pi/4$) in $V$
is collapsed to zero by $T$.
Nevertheless, the estimates for $\nm{T(x)-T(y)}$
 which we present in the following proposition hold for any $x, y$ in $V$ and will be used to get
 the necessary bilipschitz estimates when $x, y$ are in $\mm_i$ for a suitable subsequence $\lc{\Psi_i}_i$.

\begin{proposition}
\label{P:dT_est}
For any nonzero vector  $x$ in $V$,
\begin{equation}
\label{E:DTray}
\frac{\nm{T (x)}}{\nm{x}} = g\lp{\sigma, u}
\ \mbox{ where } \sigma=\tau(\nm{x})    \mbox{ and  }  u= \frac{x}{\nm{x}}.
\end{equation}
For any $x, y\in V$, $x\neq y$, $\nm{x}\ge \nm{y}>0$, the inequality
\begin{equation}
\label{E:DTest}
g(\sigma,w) - \ve  g(\theta,v)  \le \frac{  \nm{T(x) - T(y)}}{\nm{x-y}} \le g(\sigma,w) +\ve  g(\theta,v)
\end{equation}
holds,
where   $\sigma=\tau(\nm{x})$, $w =  \frac{ x -y}{\nm{x-y}}$,  $\theta=\frac{\pi}{2}+ \frac{\tau(\nm{x}) +\tau(\nm{y})}{2}$,
 and $v=  \frac{ y}{\nm{y}}$.
\end{proposition}
\proof
%To simplify notations we will use
%\[ c(x) = \cos\lp{\tau(\nm{x})} \mbox{ and } s(x) = \sin\lp{\tau(\nm{x})}.\]
The equality \eqref{E:DTray} is obvious.
To prove \eqref{E:DTest} we write
\begin{equation}
\label{E:DT}
\begin{aligned}
T(x) - T(y) = &  c(x) \lp{E x-Ey} + s(x) \lp{ Fx-Fy} \\
   & + (c(x)-c(y))Ey + (s(x)-s(y))Fy.
   \end{aligned}
\end{equation}
The first two terms give
\[ c(x) \lp{E x-Ey} + s(x) \lp{ Fx-Fy} = \lp{c(x) E+s(x)F}(x-y)= G_\sigma (x-y),\]
where $\sigma=\tau(\nm{x})$.

To estimate the sum of the last two terms in \eqref{E:DT},  note that
\[ \begin{aligned}
c(x)-c(y)  & = \cos(\tau(\nm{x}))-\cos(\tau(\nm{y})) \\
    & =- 2\sin\lp{ \frac{\tau(\nm{x}) -\tau(\nm{y})}{2}} \sin\lp{ \frac{\tau(\nm{x}) +\tau(\nm{y})}{2}},
    \end{aligned}\]
and
\[ \begin{aligned}
s(x)-s(y) & = \sin(\tau(\nm{x}))-\sin(\tau(\nm{y})) \\
   &= 2\sin\lp{ \frac{\tau(\nm{x}) -\tau(\nm{y})}{2}} \cos\lp{ \frac{\tau(\nm{x}) +\tau(\nm{y})}{2}}.
   \end{aligned}\]
 Denote $\theta=  \frac{\pi}{2}+\frac{\tau(\nm{x}) +\tau(\nm{y})}{2}$, so that
\[ \begin{aligned}
 (c(x)-c(y))Ey + & (s(x)-s(y))Fy  \\
  & =  2\sin\lp{ \frac{\tau(\nm{x})  - \tau(\nm{y})}{2}} \lp{ \cos (\theta)Ey+\sin(\theta)Fy} \\
&  = 2\sin\lp{ \frac{\tau(\nm{x}) -\tau(\nm{y})}{2}} G_\theta y.  \\
 \end{aligned}\]

Replacing back in \eqref{E:DT} we have
\begin{equation}
\label{E:DT2}
T(x) - T(y) = G_\sigma (x-y) +  2\sin\lp{ \frac{\tau(\nm{x}) -\tau(\nm{y})}{2}} G_\theta y.
\end{equation}

Since we assume that $\nm{y}\le \nm{x}$, $\tau$ is non-decreasing, and   $\tau(\nm{x})  \le \frac{\pi}{2}$,
by the  Mean Value Theorem we obtain
\[ 0 \le 2\sin\lp{ \frac{\tau(\nm{x}) -\tau(\nm{y})}{2}} \le   \tau(\nm{x}) -\tau(\nm{y}) =\tau'(t) \lp{ \nm{x}-\nm{y}},\]
for some $\nm{y}\le t\le \nm{x}$.
Using the upper bound on the derivative of $\tau$ in item~(\ref{i:Dtau}) above, we get
\[ 0 \le 2\sin\lp{ \frac{\tau(\nm{x}) -\tau(\nm{y})}{2}} \le \frac{\ve}{t} \lp{ \nm{x}-\nm{y}} \le \frac{\ve}{\nm{y}} \nm{x-y}. \]
Applying this estimate to the second term in the right hand side of \eqref{E:DT2} we obtain
\[ \begin{aligned}
\nm{ G_\sigma (x-y)} -\frac{\ve}{\nm{y}}   \nm{x-y}  \nm{ G_\theta y}  & \le \nm{T(x) - T(y)} \\
     & \le \nm{ G_\sigma (x-y)} +  \frac{\ve}{\nm{y}} \nm{ x-y}  \nm{ G_\theta y}.
\end{aligned} \]
Denote $w =  \frac{ x -y}{\nm{x-y}}$, $v=  \frac{ y}{\nm{y}}$ and thus
\[ \nm{ G_\sigma (x-y)} = g(\sigma,w)  \nm{x-y} \ \mbox{ and } \
  \nm{ G_\theta y} =  g(\theta,v)  \nm{y}.\]
We therefore obtain
\[ \lp{g(\sigma,w)  - \ve g(\theta,v) }\nm{x-y} \le \nm{T(x) - T(y)}\le \lp{g(\sigma,w)  +\ve  g(\theta,v)}\nm{x-y}. \]
\endproof
For this estimate to be useful $g(\sigma, w)$ needs to be nontrivially large so that the left hand side in \eqref{E:DTest}
is far from zero.
To make it nontrivially large we use the fact that we can select the subsequence $\lc{\Psi_i}_i$ arbitrarily.
The selection process is based on the Ramsey-type result of \cite[Proposition 1]{BS74}.
This is what we are going to describe next.

%{\color{red} At this point we need to select subsequence using Ramsey Theorem.}

\subsubsection{Brunel-Sucheston sequences}
Let $c_{00}$ denote the vector space of sequences with a finite number of nonzero terms.
 \begin{proposition}[Proposition~1 in \cite{BS74}]
\label{P:P1BS}
 Each bounded sequence $\lc{x_n}_n$ in a Banach space $X$ contains a subsequence $\lc{e_n}_n$ with the following property:
 For each $\acv =(a_1,a_2, \dots)\in c_{00}$ there exists a number $L(\acv)$ such that for every $\ve >0$,
there exists a positive integer $\nu$ with
\[ \av{ \nm{ \sum_i a_i e_{n_i}} -L(\acv)}\le \ve \]
for all integers $n_i$ with $\nu \le n_1<n_2< \cdots$.
\end{proposition}

\begin{definition}
\label{D:BSs}
We say that a sequence $\lc{e_n}_n$ in a Banach space $X$ is a {\it BS-sequence} (Brunel-Sucheston) if
for any $\acv=(a_1,a_2, \dots)\in c_{00}$ there exists a number $L(\acv)$ such that for any $\ve >0$,
there exists a positive integer $\nu$ with
\[ \av{ \nm{ \sum_i a_i e_{n_i}} -L(\acv)}\le \ve \]
for all integers $n_i$ with $\nu \le n_1<n_2< \cdots$.
\end{definition}
Therefore, Proposition~\ref{P:P1BS} states that  each bounded sequence in a Banach space
contains a subsequence which is a  BS-sequence.
It is clear that any subsequence of a BS-sequence is  a BS-sequence.

\subsubsection{Estimating $L(\acv)$ from below}
\begin{proposition}
\label{P:Pab}
 Assume $\lc{e_n}_n$ is a BS-sequence in $X$ with $\nm{e_n}\to 1$.
Let $\acv=(c, s, 0,0, \dots)\in c_{00}$ with $c, s \ge 0$, not both zero.
Then
\[L(\acv) \ge \max \lc{\frac{c(c+s)}{c+2s}, \frac{s(c+s)}{2c+s}}.\]
 \end{proposition}
 \proof
 We consider triples of ordered indices $\nu \le i<j<k$ and denote $f_{i,j} = \lp{c e_i+s e_j}$ and similarly for $f_{j,k}= \lp{c e_j+s e_k}$ and $f_{i,k}= \lp{c e_i+s e_k}$.
 As $\nu$ tends to infinity, all three  norms,
 $\nm{f_{i,j}}$, $\nm{f_{j,k}}$, and $\nm{f_{i,k}}$ tend to $L(\acv)$.
Note that
 \[ c \nm{f_{i,j}} +s \nm{f_{j,k}}+s \nm{f_{i,k}} \ge
 \nm{ c f_{i,j} - s f_{j,k}+s f_{i,k}} = \nm{ (c^2+cs)e_i}.
 \]
 Sending $\nu$ to infinity, we obtain
 \[ (c+2s) L(\acv) \ge c(c+s), \mbox{ i.e. } L(\acv) \ge \frac{c(c+s)}{c+2s}.\]
Similarly, note that
  \[ c \nm{f_{i,j}} +s \nm{f_{j,k}}+c \nm{f_{i,k}} \ge
 \nm{ -c f_{i,j} +s f_{j,k}+c f_{i,k}} = \nm{ (s^2+cs)e_k}.
 \]
  In the limit as $\nu$ tends to infinity, we obtain
 \[ (2c+s) L(\acv) \ge s(c+s), \mbox{ i.e. } L(\acv) \ge \frac{s(c+s)}{2c+s}.\]
  \endproof

  \begin{remark}
 \label{R:La_est}
 Applying Proposition~\ref{P:Pab} to $\acv=(\cos\tau, \sin\tau, 0,0, \dots)$ with $\tau \in [0,\pi/2]$,  we get
\[ L(\acv) \ge  \max \lc{\frac{\cos\tau(\cos\tau+\sin\tau)}{\cos\tau+2\sin\tau}, \frac{\sin\tau(\cos\tau+\sin\tau)}{2\cos\tau+\sin\tau}}
 \ge \frac{\sqrt{2}}{3}.\]
 The rightmost inequality follows from
 \[\max \lc{\frac{\cos\tau(\cos\tau+\sin\tau)}{\cos\tau+2\sin\tau}, \frac{\sin\tau(\cos\tau+\sin\tau)}{2\cos\tau+\sin\tau}}
=\frac{\cos\tau+\sin\tau}{\min \lc{1+2\tan \tau,1+2\cot \tau}},  \]
which equals $p(\tau) = \cos\tau\frac{1+\tan\tau}{1+2\tan \tau}$   if $0\le \tau \le \pi/4$, and
$q(\tau) = \sin\tau\frac{1+\cot\tau}{1+2\cot \tau}$   if $\pi/4 \le \tau \le \pi/2$.
It is easy to see that $p$ is decreasing on the interval $[0,\pi/4]$ with minimum value $p(\pi/4) =  \frac{\sqrt{2}}{3}$.
Similarly, $q$ is increasing on the interval $[\pi/4,\pi/2]$ with minimum value $q(\pi/4) =  \frac{\sqrt{2}}{3}$.
 \end{remark}

%\medskip

%Without loss of generality we will assume that the origin of $Y$ is a point in $M$.
\subsubsection{Subsequence selection}

For each $i \ge 1$, recall  the notation
$$\mm_i=\lc{ x \in \mm \ :   \  R_{i-1} < \nm{x} \le r_{i+1}} \cup \lc{0},$$
 with the convention that $R_0=0$.
Consider the set $U_i$ of all unit vectors
\[ U_i=\lc{ u = \frac{x-y}{\nm{x-y}}  : x, y\in \mm_i,  r_i < \nm{x},  \nm{x} \ge \nm{y}, x\neq y , \nm{y} \le R_i }.\]
Since $\mm$ is locally finite,  each  of the sets $\mm_i$, and therefore $U_i$, is finite.
Since a vector $u$ in $U_i$ can be the direction of multiple differences
$(x_j-y_j)$ of points in $\mm_i$ with  $\nm{x_j} \ge \nm{y_j}$, we introduce a set $\mt_i(u)$  (see below)
in which we record
all numbers
$\tau_i(\nm{x_j})$ (see items  \ref{i:t1}, \ref{i:t2}, \ref{i:t3} in Subsection~\ref{ss:Cf} for the definition of $\tau_i$).
So, for each $u \in U_i$ we define the (finite) set $\mt_i (u)$ consisting of all  $\tau_i\lp{\nm{x}} $ for $x\in \mm_i$
satisfying the following conditions:
\begin{itemize}
\item $r_i < \nm{x}$,
\item $ \exists y \in \mm_i, \nm{y} \le R_i, \nm{x} \ge \nm{y}, x\neq y ,
\frac{x-y}{\nm{x-y}}=u$.
\end{itemize}

 \begin{proposition}
\label{P:UBS}
For any $\zeta> 0$,
there exists a subsequence $\lc{\Psi_{n}}_n$ of $\lc{\psi_{n}}_n$
 so that for each $i\ge 1$, for each point $u\in U_i$, and for each $\tau \in \mt_i (u)$,
we have
 \[ \nm{\cos \tau \Psi_i (u) +\sin \tau \Psi_{i+1} (u) } \ge \frac{\sqrt{2}}{3(1+\zeta)}.\]
\end{proposition}
\proof

Let $U_1= \lc{u_1, u_2,  \dots, u_c}$.
By Proposition~\ref{P:P1BS}, the sequence  $\lc{\psi_{n}(u_1)}_{n}$
contains a subsequence  $\lc{\psi_{n}^0(u_1)}_{n}$
which is a BS-sequence.
Let  $\tau_j \in \mt_1 (u_1) = \lc{\tau_1, \dots, \tau_b}$ and $\acv=\lp{\cos \tau_j, \sin \tau_j, 0, \dots}$.
Since $\lc{\psi_{n}^0(u_1)}_{n}$  is a BS-sequence and from Remark~\ref{R:La_est} we have
$L(\acv) \ge \frac{\sqrt{2}}{3}$, there exists
a positive integer $\nu_j$ so that
\[ \nm{\cos \tau_j \psi_{n_1}^0 (u_1) +\sin \tau_j \psi_{n_2}^0 (u_1) } \ge \frac{\sqrt{2}}{3(1+\zeta)}\]
for all $\nu_j \le n_1 < n_2$.
Let $\nu = \max_{1\le j \le b} \nu_j$ and relabel the sequence $\lc{\psi_{n}^0}_{n= \nu}^\infty$
by $\lc{\psi^0_{1,n}}_{n= 1}^\infty$.

We now repeat the process for the sequence  $\lc{\psi^0_{1,n}(u_2)}_{n}$ to obtain a new  sequence
$\lc{\psi^0_{2,n}}_{n}$, and so on until we reach $\lc{\psi^0_{c,n}}_{n}$.
At this point, we relabel the  last sequence $\lc{\psi^0_{c,n}}_{n}$ by $\lc{\psi^1_{n}}_{n}$.

Note that for any $u\in U_1$ and for any $\tau \in \mt_1(u)$ we have
that
\[ \nm{\cos \tau \psi^1_{n_1} (u) +\sin \tau \psi^1_{n_2} (u) } \ge \frac{\sqrt{2}}{3(1+\zeta)}\]
for all $1 \le n_1 < n_2$.

After that we proceed by induction. For $i\ge 2$ we select from the sequence $\lc{\psi^{i-1}_n}_n$
a subsequence $\lc{\psi^i_n}_n$ so that for any $u\in U_i$ and for any $\tau \in \mt_i(u)$ we have
that
\[ \nm{\cos \tau \psi^i_{n_1} (u) +\sin \tau \psi^i_{n_2} (u) } \ge \frac{\sqrt{2}}{3(1+\zeta)}\]
for all $1 \le n_1 < n_2$.

 The desired subsequence  $\lc{\Psi_{n}}_n$  is the diagonal $\lc{\psi^n_{n}}_n$.
\endproof

\begin{remark}
Note that the resulting sequence $\lc{\Psi_{n}}_n$ from Proposition~\ref{P:UBS} depends on the selected parameter $\zeta$
and therefore so does the function $\Psi$ defined from this sequence via formula \eqref{E:DefPsi}.
\end{remark}
\medskip

We now combine  Proposition~\ref{P:dT_est} and Proposition~\ref{P:UBS} to obtain the bilipschitz estimates for
$\Psi$ restricted to each of the sets $\mm_i$.% of $\mm$.
\begin{proposition}
\label{P:dTmi} Fix $i\ge 1$.
For any  $x \in \mm_i$ satisfying $r_i < \nm{x}$,
\begin{equation}
\label{E:DTray_mi}
\frac{\sqrt{2}}{3(1+\zeta)} \le \frac{\nm{\Psi (x)}}{\nm{x}}  \le \sqrt{2} (1+\gamma).
\end{equation}
For any $x, y\in \mm_i$, $x\neq y$, $\nm{x}\ge \nm{y}$ satisfying  $\nm{y} \le R_i$,  $r_i < \nm{x}$,
\begin{equation}
\label{E:DTest_mi}
\frac{\sqrt{2}}{3(1+\zeta)} - \ve  \sqrt{2} (1+\gamma) \le \frac{  \nm{\Psi(x) - \Psi(y)}}{\nm{x-y}}
\le \sqrt{2}(1+\gamma)+\ve  \sqrt{2} (1+\gamma).
\end{equation}
\end{proposition}
\proof
We recall  that on $\mm_i$ the map  $\Psi$  is given by the two-term formula \eqref{E:PsiMi}.
 It is clear that for  the map $T$ defined by \eqref{E:Tdef} we have
 $\left. \Psi\right|_{\mm_i}= \left. T\right|_{\mm_i}$.

Observe that \eqref{E:DTray} implies
\[ \frac{\nm{\Psi (x)}}{\nm{x}}  = g\lp{\sigma, u}
\ \mbox{ where } \sigma=\tau(\nm{x})    \mbox{ and  }  u= \frac{x}{\nm{x}}.\]
For  $x\in \mm_i$ satisfying $r_i < \nm{x}$, we have that $u\in U_i$ and $\sigma \in \mt_i(u)$.
Therefore, Proposition~\ref{P:UBS}  and Remark~\ref{R:beta_est} imply
$\frac{\sqrt{2}}{3(1+\zeta)} \le g\lp{\sigma, u} \le  \sqrt{2} (1+\gamma). $

Similarly, for $x, y\in \mm_i$, $x\neq y$, $\nm{x}\ge \nm{y}$ satisfying $\nm{y} \le R_i$,  $r_i < \nm{x}$,
Proposition~\ref{P:UBS}  and Remark~\ref{R:beta_est} imply
\[ \frac{\sqrt{2}}{3(1+\zeta)} \le g\lp{\sigma, w} \le  \sqrt{2} (1+\gamma),\]
where   $\sigma=\tau(\nm{x})$, $w =  \frac{ x -y}{\nm{x-y}}$.

With  $\theta=\frac{\pi}{2}+ \frac{\tau(\nm{x}) +\tau(\nm{y})}{2}$,
 and $v=  \frac{ y}{\nm{y}}$  we only use the estimate from above given by Remark~\ref{R:beta_est},
 namely  $g\lp{\theta, v} \le  \sqrt{2} (1+\gamma)$.
Combining the estimates for  $g\lp{\sigma, w}$ and $g\lp{\theta, v}$ in \eqref{E:DTest}
we obtain the  inequalities \eqref{E:DTest_mi}.
\endproof

\begin{remark}
\label{R:i_indep}
Note that the estimates in Proposition~\ref{P:dTmi}, which hold for $x,y \in \mm_i$, are independent of $i\ge 1$.
\end{remark}
\begin{remark}
\label{R:ann_est}
A second observation is that for $x,y \in \mA_{R_{i-1},r_i}$
much better estimates are available than the ones in Proposition~\ref{P:dTmi}
because on this set $\Psi$ coincides with $\Psi_i$ and therefore
\[ 1\le  \frac{\nm{\Psi (x)}}{\nm{x}}  \le 1+\gamma
\mbox{ and } 1\le  \frac{  \nm{\Psi(x) - \Psi(y)}}{\nm{x-y}}\le 1+\gamma. \]
\end{remark}

\subsection{End of the proof of Theorem~\ref{T:main}}

Let $x,y\in \mm$, with $\nm{x} \ge \nm{y}$.
If there exists an $i \ge 1$ so that both $x, y$ are in the same $\mm_i$, then we use  the estimates in Proposition~\ref{P:dTmi},
or the sharper estimates in Remark~\ref{R:ann_est}.

Otherwise, necessarily there exist $1\le i \le j$ so that  $y \in \mA_{R_{i-1},R_i}$ and  $x\in \mA_{r_{j+1},r_{j+2}}$.
In this case  $\|y\|\le R_i$  and $ r_{i+1} < \|x\|$.
From \eqref{E:rR_ratio} we obtain that
\begin{equation}
\label{E:delta}
 \nm{y} \le \delta \nm{x},
 \end{equation}
  and therefore
\[(1-\delta) \nm{x} = \nm{x}-\delta\nm{x} \le \nm{x}-\nm{y} \le \nm{x-y}. \]
Similarly
\[ \nm{x-y} \le \nm{x}+\nm{y} \le \nm{x}+\delta\nm{x}=(1+\delta)\nm{x}.\]
Therefore
\begin{equation}
\label{E:delo}
 \frac{ \nm{x-y} }{1+\delta} \le \nm{x} \le  \frac{ \nm{x-y} }{1-\delta}.
  \end{equation}
Next, in the inequalities
\[ \nm{\Psi(x)} -\nm{\Psi(y)}\le  \nm{\Psi(x) -\Psi(y)} \le  \nm{\Psi(x)} +\nm{\Psi(y)}, \]
we use \eqref{E:DTray_mi} for $\nm{\Psi(x)}$ and $\nm{\Psi(y)}$ to obtain
\[ \frac{\sqrt{2}}{3(1+\zeta)} \nm{x} - \sqrt{2} (1+\gamma)\nm{y}\le  \nm{\Psi(x) -\Psi(y)} \le
\sqrt{2} (1+\gamma)\nm{x}+\sqrt{2} (1+\gamma)\nm{y}. \]
Using again \eqref{E:delta} we obtain
\[ \lp{\frac{\sqrt{2}}{3(1+\zeta)}- \sqrt{2} (1+\gamma)\delta}\nm{x}\le  \nm{\Psi(x) -\Psi(y)} \le
\sqrt{2} (1+\gamma)(1+\delta)\nm{x}. \]
Further combining with \eqref{E:delo} this yields
\begin{equation}
\label{E:del_far}
\lp{\frac{\sqrt{2}}{3(1+\zeta)}- \sqrt{2} (1+\gamma)\delta}  \frac{1}{1+\delta}
\le  \frac{ \nm{\Psi(x) -\Psi(y)} }{ \nm{x-y} }
\le \sqrt{2} (1+\gamma)\frac{1+\delta}{1-\delta}.
  \end{equation}
It is clear that for a given $\epsilon>0$,
by selecting a priori sufficiently small constants $\ve$, $\delta$, $\gamma$, and $\zeta$ we get that
in both \eqref{E:DTest_mi} and \eqref{E:del_far},  the left hand sides are at least $\ds \frac{\sqrt{2}}{3\sqrt{1+\frac{\epsilon}{3}}}$,
while the right hand sides are at most $\ds \sqrt{2}\sqrt{1+\frac{\epsilon}{3}}$.
Since the ratio of the upper bound over  the lower bound is
 $(3+\epsilon)$, this concludes the proof of Theorem~\ref{T:main}.

\section{Some consequences of Theorem \ref{T:main}}\label{S:FD}

The goal of this section is to derive from Theorem \ref{T:main}
the following corollary, which contains finite determination
results for bilipschitz and coarse embeddings, as well as similar
finite determination results for other types of embeddings which
one can define. This argument is essentially known (see
\cite{Ost12}, \cite[Sections~2.3.2-2.3.3]{Ost13}), we present it
here for convenience of the readers. Necessary background is
contained in \cite[Section 2.2]{Ost13}.

\begin{corollary}\label{C:FDCoarse}
Let $(\mm,d) $ be a locally finite metric space and $X$ be a
Banach space. Suppose that $\mm$ and $X$ satisfy the following
condition: there exist non-decreasing functions
$\rho_1,\rho_2:[0,\infty)\to[0,\infty)$ (observe that this
condition implies that $\rho_1$ and $\rho_2$ have finite values)
such that $\lim_{t\to\infty}\rho_1(t)=\infty$, and for each finite
subset $\mn$ of $\mm$ there is a map $f_\mn:\mn\to X$ satisfying
the condition
\begin{equation}\label{E:CorN}\forall u,v\in \mn~ \rho_1(d_\mm(u,v))\le
\|f_\mn(u)-f_\mn(v)\|\le\rho_2(d_\mm(u,v)).\end{equation} Then,
for every $\epsilon>0$ there exists a mapping $F_\mm:\mm\to X$
satisfying
\begin{equation}\label{E:CorM}\forall u,v\in \mm~ \rho_1(d_\mm(u,v))\le
\|F_\mm(u)-F_\mm(v)\|\le(3+\epsilon)\rho_2(d_\mm(u,v)).\end{equation}
\end{corollary}

\begin{proof} Pick $O\in\mm$, restrict attention only to finite sets $\mn\ni
O$ and denote by $\ms$ the set of all such $\mn$. Assume that
$f_\mn(O)=0$ for each $\mn\in\ms$. Let $\mf$ be the filter generated
by $\mr=\{\{\mcp: \mcp\supset\mn\}:\mn\in\ms\}$ and $\mju$ be an
ultrafilter on $\ms$ containing $\mr$. Then the ultraproduct of
maps $f_\mn$ determines a map $f_\mm:\mm\to X^\mju$, where
$X^\mju$ is the ultrapower of $X$. This map satisfies
\begin{equation}\label{E:CorMno3}\forall u,v\in \mm~ \rho_1(d_\mm(u,v))\le
\|f_\mm(u)-f_\mm(v)\|\le\rho_2(d_\mm(u,v)).\end{equation}

The condition $\lim_{t\to\infty}\rho_1(t)=\infty$ implies that the
image $\mm':=f_\mm(\mm)$ is a locally finite subset in $X^\mju$.
As is well-known (\cite[Proposition 2.31]{Ost13}), $X^\mju$ is
finitely represented in $X$. By Theorem \ref{T:main}, there is a
map $g:\mm'\to X$ satisfying
\begin{equation}\label{E:Cor3}\forall u',v'\in \mm'~ \|u'-v'\|_{X^\mju}\le
\|g(u')-g(v')\|\le(3+\epsilon)\|u'-v'\|_{X^\mju}.\end{equation}

It is easy to see that the composition $F_\mm=gf_\mm$ is the
desired map.
\end{proof}

Corollary \ref{C:FDCoarse} shows that coarse embeddability is
finitely determined. If we apply it to more narrow classes of
functions satisfying the condition: positive real multiples of a
function are also in the class, we get that that the corresponding
types of embeddability are also finitely determined.

\section{Related open problems}

In our opinion, the most important open problem related to the topic of
this paper is:

\begin{problem} Can one replace $(3+\epsilon)$ by $(1+\epsilon)$
in Theorem \ref{T:main}?
\end{problem}

This problem has been already raised in \cite[Problem 5.1]{OO19}.
Both positive or negative answer would be very interesting.

Another interesting problem is mentioned in Remark \ref{R:Isom}.
For that problem it is not clear whether it is realistic to
describe completely the corresponding class of Banach spaces.

\subsection*{Acknowledgement}

The first-named author wants to thank his home institution, St.
John's University, for providing him a research leave in Spring
2023. The second-named author gratefully acknowledges the support
by the National Science Foundation grant NSF DMS-1953773 and
St. John's University Summer 2023 Support of Research program.

%\end{large}

\renewcommand{\refname}{

\textsc{Department of Mathematics and Computer Science, St. John's
University, 8000 Utopia Parkway, Queens, NY 11439, USA} \par
  \textit{E-mail address}: \texttt{catrinaf@stjohns.edu} \par
  \medskip

\textsc{Department of Mathematics and Computer Science, St. John's
University, 8000 Utopia Parkway, Queens, NY 11439, USA} \par
  \textit{E-mail address}: \texttt{ostrovsm@stjohns.edu} \par

\end{document}